\documentclass[12pt, english, a4paper]{amsart} 

\usepackage{graphicx}
\usepackage{babel}
\usepackage{amsmath}
\usepackage{amsthm}
\usepackage{amssymb}
\usepackage{amsfonts}
\usepackage{latexsym}
\usepackage[all, cmtip]{xy}
\usepackage[left=2.6cm, top=3cm, bottom=3cm, right=2.6cm]{geometry}

\newtheoremstyle{thm}{}{}{\slshape}{}{\scshape}{.}{0.5em}{}

\newtheoremstyle{def}{}{}{}{}{\scshape}{.}{0.5em}{}

\newtheoremstyle{rmk}{}{}{}{}{\scshape}{.}{0.5em}{}

\newtheoremstyle{claim}{}{}{}{}{\slshape}{.}{0.5em}{}

\theoremstyle{thm}
\newtheorem{newstatement}{newstatement}

\newtheorem{theorem}[newstatement]{Theorem}
\newtheorem{corollary}[newstatement]{Corollary}
\newtheorem{proposition}[newstatement]{Proposition}

\theoremstyle{def}
\newtheorem{definition}[newstatement]{Definition}

\theoremstyle{rmk}
\newtheorem{remark}[newstatement]{Remark}

\theoremstyle{claim}
\newtheorem*{claim}{Claim}

\expandafter\let\expandafter\oldproof\csname\string\proof\endcsname
\let\oldendproof\endproof
\renewenvironment{proof}[1][\proofname]{%
  \oldproof[\slshape #1]%
}{\oldendproof}

\def\fieldstyle{\mathbb} 

\def\p_#1{p_{\kern-1.5pt_#1}}
\def\emph#1{{\em #1}}
\def\textbf#1{{\bf #1}}
\def\MR#1{\relax}

\newcommand{\Int}{\mathop{\mathrm{Int}}\nolimits} 
\let\Bd\partial

\newcommand{\id}{\mathop{\mathrm{id}}\nolimits}

\newcommand{\diff}{\mathop{\mathrm{Diff}}\nolimits}

\newcommand{\M}{\mathop{\cal M}\nolimits}

\newcommand{\Sign}{\mathop{\mathrm{Sign}}\nolimits}
\newcommand{\SO}{\mathop{\mathrm{SO}}\nolimits}

\newcommand{\R}{\fieldstyle R}
\newcommand{\C}{\fieldstyle C}

\newcommand{\Z}{\mathbb Z}

\let\hash\#
\renewcommand{\#}{\mathbin{\hash}}

\renewcommand{\em}{\sl}

\def\(#1){({\em #1\/})}
\def\){\relax}

\def\varemptyset{{\text{\raise.21ex\hbox{$\not$}}\mkern.15mu\mathrm{O}\mkern.15mu}}

\let\emptyset\varemptyset
\let\geq\geqslant
\let\leq\leqslant

\let\phi\varphi
\let\epsilon\varepsilon

\makeatletter
\renewcommand{\section}{\@startsection%
{section}
{1}
{0mm}
{1.5\bigskipamount}
{0.5\bigskipamount}
{\centering\normalsize\sc}}

\renewcommand{\paragraph}{\@startsection%
{paragraph}
{4}
{0mm}
{\bigskipamount}
{-1.25ex}
{\normalsize\sl}}

\def\provedboxcontents#1{$\square$}

\makeatother

\makeatletter

\def\ft{\@ifnextchar[{\ft@s}{\ft@}}
\def\ft@{\ft@@@s[\f@size]}
\def\ft@s[{\@ifnextchar{a}{\ft@sz[}{\ft@@s[}}
\def\ft@@s[{\@ifnextchar{s}{\ft@sz[}{\ft@@@s[}}
\def\ft@@@s[#1]{\ft@sz[at #1pt]}
\def\ft@sz[#1]#2{\font\fonttemp=#2 #1\fonttemp\ignorespaces}

\makeatother

\let\co\colon

\begin{document}

\def\calstyle#1{{\text{\ft{eusm10}#1}}}

\let\cal\calstyle

\title[Universal Lefschetz fibrations]{\bf\large UNIVERSAL LEFSCHETZ FIBRATIONS\\[4pt]
AND LEFSCHETZ COBORDISMS}

\author{Daniele Zuddas}

\address{\normalfont Korea Institute for Advanced Study, 85 Hoegiro, Dongdaemun-gu, Seoul 130-722, Republic of Korea.}

\email{zuddas@kias.re.kr}

\date{}


\begin{abstract}
We construct universal Lefschetz fibrations, that are defined in analogy with the classical universal bundles. We also introduce the cobordism groups of Lefschetz fibrations, and we see how these groups are quotient of the singular bordism groups via the universal Lefschetz fibrations.
	
\end{abstract}
\maketitle

\section*{Introduction}

Topological Lefschetz fibrations over a surface have been given considerable attention in the last decade, because of their applications to symplectic and contact topology, see for example \cite{D99, LP01, AO01, GS99}. This led to several generalizations, including achiral Lefschetz fibrations and their relations with branched coverings and braided surfaces \cite{APZ2013, Z09, Fu00}, broken Lefschetz fibrations \cite{GK07, B08, B09}, and Morse 2-functions \cite{GK12, GK14}. In \cite{DSKZ2014} Matsumoto's torus fibration on $S^4$ \cite{Mat82} (see also \cite[Example 8.4.7]{GS99}) has been used for constructing an almost complex structure on $\R^4$ containing holomorphic tori.

We are going to further generalize Lefschetz fibrations by allowing the base manifold to have arbitrary dimension. The critical image of a Lefschetz fibration is a codimension-2 submanifold of the target manifold, and the monodromy is a homomorphism to the mapping class group of the fiber. Actually, to understand the larger amount of information that a generalized Lefschetz fibration carries with respect to a standard one (that is, over a surface), we need several types of monodromies, each one capturing some aspects, but not others.

Universal Lefschetz fibrations have been introduced in \cite{Z11} in analogy with the universal bundles, under the additional assumption for the base surface to have non-empty boundary. The purpose of the present paper is twofold: to relax this restriction by allowing the base surface to be closed, and to start building a (co)bordism theory for Lefschetz fibrations, along the lines of the classical bordism theory. Main results include: a characterization of universal Lefschetz fibrations in dimension two (Theorem~\ref{universal/thm}) and three (Theorem~\ref{universal3/thm}), an explicit construction of these fibrations, and an application to Lefschetz cobordism groups (that are defined in Section~\ref{lf-cob/sec}), proving that these groups are quotients of certain singular bordism groups in dimension two and three (Proposition~\ref{cobord/thm} and Corollary~\ref{lfcobord-bord/thm}). We will give some computations, and further developments, in a forthcoming paper.

Throughout this paper all manifolds and maps are assumed to be smooth. We consider only oriented compact manifolds and (local) diffeomorphisms that preserve the orientations, if not differently stated.

\paragraph{Acknowledgements} The work presented in this paper has been partially carried out at the Max Planck Institute for Mathematics in Bonn, Germany, during the 4-manifolds semester (January -- June 2013). So, I would like to thank MPIM for support and hospitality. Thanks also to Mark Grant (Newcastle University) for a useful answer in MathOverflow.

\section{Definitions, preliminaries and notations}

By the standard definition, a Lefschetz fibration is, roughly speaking, a smooth map over a surface with only non-degenerate (possibly achiral) complex singularities. In order to state our results we propose the following generalization.

For  $f \co V \to M$ we denote by $\widetilde A_f \subset V$ the critical set of $f$, and by $A_f = f(\widetilde A_f)$ the critical image of $f$.

\begin{definition}\label{lefschetz-map/def}
	Let $M$ and $V$ be manifolds of dimensions respectively $m + 2$ and $m + 2k$ with $m \geq 0$ and $k \geq 2$. A Lefschetz fibration $f \co V \to M$ is a map such that:
	
	\begin{enumerate}
		\item
			near any critical point $\tilde a \in \widetilde A_f$, $f$ is locally equivalent to the map $f_0 \co \R^m_+ \times \C^k \to \R^m_+ \times \C$ defined by $f_0(x, z_1, \dots, z_k) = (x, z_1^2 + \cdots + z_k^2)$ for $x \in \R^{m}_+ = \{(x_1, \dots, x_m) \in \R^m \;|\; x_m \geq 0\}$ and $(z_1, \dots,  z_k) \in \C^k;$
			
			\item
			$f_| \co \widetilde A_f \to M$ is an embedding;
		
		\item
			$f_| \co V - f^{-1}(A_f) \to M - A_f$ is a locally trivial bundle with fiber a manifold $F$ (the regular fiber of $f)$.
	\end{enumerate}
\end{definition}

Note that in case $A_f = \emptyset$, $f$ is an honest bundle.

\begin{definition}
	We call $f_| \co V - f^{-1}(A_f) \to M - A_f$ the regular bundle associated with $f$.
\end{definition}

We see below that there is also a {\em singular bundle} associated with $f$.
The following proposition is a simple consequence of the definition.

\begin{proposition}
	Let $f \co V \to M$ be a Lefschetz fibration. Then
	\begin{enumerate}
		\item
			$\widetilde A_f$ is a proper submanifold of $V$ of dimension $m$;
			
		\item
			$A_f$ is a proper submanifold of $M$ of codimension two;
			
		\item
			$f_| \co \widetilde A_f \to A_f$ is a diffeomorphism;
		
		\item
			the regular fiber $F \subset V$ is a submanifold of dimension $2k - 2$.
	\end{enumerate}
\end{proposition}

For $m = 0$, $f$ is an ordinary (possibly achiral) Lefschetz fibration. So, a generalized Lefschetz fibration looks locally as an ordinary one times an identity map.
Throughout the paper we assume $k = 2$. This implies that $F$ is a surface.

In general, $A_f$ can be non-orientable. However, if $A_f$ is orientable, by fixing an orientation on it (hence on $\widetilde A_f$ via $f_| \co \widetilde A_f \to A_f$) we can define the {\em positive} and the {\em negative} critical points and values: $\tilde a \in \widetilde A_f$ is a {\em positive critical point} of $f$ if the local coordinates considered in the definition can be chosen to be compatible with the orientations of $V$, $M$, and $A_f$ (that corresponds to $\R^m \times \{0\} \subset \R^m \times \C$ in the local chart). Otherwise, $\tilde a$ is said to be a {\em negative critical point}. Accordingly, $a = f(\tilde a)$ is said to be a {\em positive} or {\em negative critical value}. This positivity or negativity is locally 
invariant, hence the connected components of $A_f$ inherit it.

Two Lefschetz fibrations $f_1 \co V_1 \to M_1$ and $f_2 \co V_2 \to M_2$ are said to be {\em equivalent} if there are orientation-preserving diffeomorphisms $\phi \co V_1 \to V_2$ and $\psi \co M_1 \to M_2$ such that $\psi \circ f_1 = f_2 \circ \phi$. This implies that $\psi(A_{f_1}) = A_{f_2}$ and that $\phi(\widetilde A_{f_1}) = 
\widetilde A_{f_2}$. If $A_{f_1}$ and $A_{f_2}$ are oriented, we assume also that $\psi_| \co A_{f_1} \to A_{f_2}$ is orientation-preserving. If $f_1$ and $f_2$ are equivalent, we make use of the notation $f_1 \cong f_2$.

Let $f \co V \to M$ be a Lefschetz fibration with regular fiber $F = F_{g,b}$, the oriented surface of genus $g$ with $b$ boundary components, and let $N$ be a $n$-manifold.

\begin{definition}\label{f-regular/def}
	A map $q \co N \to M$ is said to be $f$-regular if $q$ and $q_{|\Bd N}$ are transverse to $f$.
\end{definition}

If $q \co N \to M$ is $f$-regular then $\widetilde V = \{(x, v) \in N \times V \;|\; q(x) = f(v)\}$ is a $(n + 2)$-manifold and the map $\widetilde f \co \widetilde V \to N$ defined by $\widetilde f(x, v) = x$ is a Lefschetz fibration. The map $\widetilde q \co \widetilde V \to V$ defined by $\widetilde q(x, v) = v$ sends each fiber of $\widetilde f$ diffeomorphically onto a fiber of $f$, hence the regular fiber of $\widetilde f$ is still $F$.
Moreover, we have $A_{\widetilde f} = q^{-1}(A_f)$.

\begin{definition}\label{lefschetz-pul/def}
	We say that $\widetilde f$ is the pullback of $f$ by $q$. We denote it by $\widetilde f = q^*(f)$.
\end{definition}

Let $\cal L(F)$ be some class of Lefschetz fibrations with fiber $F$.

\begin{definition}\label{univ/def}
	We say that a Lefschetz fibration $u \co U \to M$ with fiber $F$ is $\cal L(F)$-universal (or universal with respect to $\cal L(F))$ if $(1)$ for any $f \co V \to N$ that belongs to $\cal L(F)$ there exists a $u$-regular map $q \co N \to M$ such that $q^*(u) \cong f$, and $(2)$ any such pullback for an arbitrary $q \co N \to M$ belongs to $\cal L(F)$ up to equivalence, where $N$ is the base of a Lefschetz fibration of $\cal L(F)$.
\end{definition}

In other words, $u$ is $\cal L(F)$-universal if and only if the class $\cal L(F)$ coincides with the class of pullbacks of $u$ obtained by those maps $q \co N \to M$ such that $N$ is the base of a Lefschetz fibration that belongs to $\cal L(F)$.

\paragraph{Monodromies}
Now we consider connected Lefschetz fibrations. The non-connected ones con be easily handled by considering the restrictions to the connected components.

Let $\M_{g,b}$ be the mapping class group of $F_{g,b}$, namely the group of self-diffeomorphisms of $F_{g,b}$ which keep $\Bd F_{g,b}$ fixed pointwise, up to isotopy through such diffeomorphisms. Let also $\widehat\M_{g,b}$ be the general mapping class group of $F_{g,b}$, whose elements are the isotopy classes of orientation-preserving self-diffeomorphisms of $F_{g,b}$ (without assumptions on the boundary).

The regular bundle associated with $f \co V \to M$ has a monodromy homomorphism $\widehat\omega_f \co \pi_1(M - A_f) \to \widehat \M_{g,b}$.

\begin{definition}
	We call $\widehat\omega_f$ the bundle monodromy of $f$.
\end{definition}

For a codimension-2 submanifold $A \subset M$, let $N(A)$ be a compact tubular neighborhood of $A$ in $M$, endowed with its disk bundle structure $B^2 \hookrightarrow N(A) \to A$. Take the base point $* \in M - N(A)$, and let $N(*) \subset M - N(A)$ be a small ball around $*$. We join $N(*)$ with each component of $N(A)$ by a narrow 1-handle, and let $\overline N(A)$ be the result. By construction, the manifold $\overline N(A)$ is uniquely determined, up to diffeomorphisms, by the normal bundle of $A$ in $M$, although its embedding in $M$ in general is not unique. If $A$ is connected, we have $\overline N(A) \cong N(A)$.

We denote by $\mu_1(M, A)$ the subgroup of $\pi_1(M - A)$ generated by the meridians of $A$ in $M$. Note that $\mu_1(M, A)$ is the kernel of the homomorphism induced by the inclusion $i_* \co \pi_1(M - A) \to \pi_1(M)$, so it is a normal subgroup.

Let $f \co V \to M$ be a Lefschetz fibration, and let $\bar f \co \overline V \to \overline N(A_f)$ be the restriction of $f$ over $\overline N(A_f)$.

Taking a fiber of $N(A_f) \to A_f$, that is a transverse 2-disk $B^2$, the restriction of $f$ over it is a Lefschetz fibration $f' \co V' \to B^2$ with only one critical point. So, its monodromy is a Dehn twist \cite{GS99} about a curve $c \subset F$, which is said to be a {\em vanishing cycle}. Thus, the singular fiber is homeomorphic to $F / c$. The vanishing cycles that correspond to different components of $A_f$ might be topologically different as embedded curves in $F$. However, the local model of $f$ near a critical point implies that the restriction of $f$ over a component of $A_f$ is a locally trivial bundle over that component with fiber $F/c$ (the total space is not a topological manifold).

\begin{definition}
	We call $f_| \co f^{-1}(A_f) \to A_f$ the singular bundle associated with $f$.
\end{definition}

Note that the singular fiber $F/c$ is homeomorphic to a (possibly disconnected) surface $F_c$, with two points $p_1$ and $p_2$ identified. The surface $F_c$ is obtained by surgering $F$ along $c$. Moreover, any self-homeomorphism of $F/c$ lifts to a unique homeomorphism of $(F_c, \{p_1, p_2\})$. If $c$ is non-separating, then $F_c \cong F_{g-1, b}$ is connected, so in this case the monodromy of the singular bundle is a homomorphism $\omega_f^{\Join} \co \pi_1(A_f) \to \widehat\M_{g-1, b, 2}$, where $\widehat\M_{g, b, n}$ denotes the general mapping class group of $F_{g,b}$ with $n$ marked points (mapping classes are allowed to permute the marked points and the boundary components). In general, we have to consider the general mapping class group of a surface with two marked points and with at most two components (each one containing a marked point).

\begin{definition}
	We call $\omega_f^\Join$ the singular monodromy of $f$. If $A_f$ is not connected, we intend that $\omega_f^\Join$ is the collection of the singular monodromies of the components of $A_f$.
\end{definition}

\begin{remark}
	If the vanishing cycles are all non-separating and $g \geq 2$, the singular bundle is determined by the singular monodromy.
\end{remark}

We already know that the monodromy of a meridian of $A_f$ in $\overline N(A_f)$ is a Dehn twist. Then, there is a canonical homomorphism $\omega_f \co \mu_1(\overline N(A_f), A_f) \to \M_{g,b}$ that sends a meridian to the corresponding Dehn twist.

\begin{definition}
	We call $\omega_f$ the Lefschetz monodromy of $f$.
\end{definition}

We say that $f$ is an {\em allowable Lefschetz fibration} if the monodromy of an arbitrary meridian of $A_f$ is a Dehn twist about a curve $c \subset F$ that is homologically essential in $F$. For the sake of simplicity, we assume that the Lefschetz fibrations we consider are allowable, if not differently stated. However, most results of this paper hold also in the non-allowable case, by suitable modifications.

Consider the canonical homomorphism $\beta \co \M_{g,b} \to \widehat\M_{g,b}$ that sends a mapping class $[\phi] \in \M_{g,b}$ to the mapping class $[\phi] \in \widehat\M_{g,b}$.

The Lefschetz and the bundle monodromies are related by the following commutative diagram
\begin{equation}\label{cd1/eqn}
	\xymatrix{
	\mu_1(\overline N(A_f), A_f)  \ar[r]^{i_*} \ar[d]_{\omega_f} & \pi_1(M - A_f) \ar[d]^{\widehat\omega_f} \\
	\M_{g,b} \ar[r]_{\beta} & \widehat\M_{g,b}}
\end{equation}
where $i_*$ is induced by the inclusion $i\co \overline N(A_f) - A_f \hookrightarrow M - A_f$.

There is also a compatibility condition between the bundle monodromy and the singular monodromy. Roughly speaking, the monodromy of a loop contained in $N(A_f)$ must preserve the vanishing cycle associated to this component.

Let $\Pi_1(F) = \pi_1(\diff(F), \id)$.
We say that $F$ is {\em exceptional} if $\Pi_1(F) \ne 0$. It is known that $F_{g,b}$ is exceptional if and only if $(g,b) \in \{(0, 0), (0, 1), (0, 2), (1, 0)\}$, see for example \cite{Gr73}. However, an allowable Lefschetz fibration with fiber the disk or the sphere is necessarily an honest bundle, and for this reason we assume that the fiber is not the sphere or the disk. So, the only exceptional fibers we admit are the torus and the annulus. Moreover, for any $(g,b) \ne (0, 0)$, $\pi_i(\diff(F_{g,b}), \id) = 0$ for all $i > 1$, $\Pi_1(T^2) \cong \Z^2$, and $\Pi_1(S^1 \times I) \cong \Z$, see \cite{EE67, EE69, Gr73}.

In order to state our results, we need a further invariant of Lefschetz fibrations. 
Consider an element $[\alpha] \in \pi_2(M - A_f)$, $\alpha \co S^2 \to M - A_f$, and let $\widetilde f = \alpha^*(f) \co \widetilde V \to S^2$. It follows that $\widetilde f$ is a locally trivial $F$-bundle. Decompose $S^2 = D_1 \cup_\partial D_2$ as the union of two disks $D_1$ and $D_2$, and trivialize $\widetilde f$ over $D_i$, that is $\widetilde f^{-1}(D_i) \cong D_i \times F$. The two trivializations differ by an element $\tau \in \Pi_1(F)$ along $\partial D_1 = \partial D_2 \cong S^1$. This defines a homomorphism $\omega_f^s \co \pi_2(M - A_f) \to \Pi_1(F)$, such that $\omega_f^s([\alpha]) = \tau$. This homomorphism is nothing but the one that fits in the homotopy exact sequence of the associated $\diff(F)$-bundle over $M - A_f$.

\begin{definition}
	We call $\omega_f^s$ the structure monodromy of $f$.
\end{definition}

Now, consider the pullback $\widetilde f = q^*(f)$, with $f \co V \to M$ and $q \co N \to M$ being base point preserving. Let $q_* \co \pi_i(N) \to \pi_i(M)$, and let $q_{|*} \co \pi_{i}(N - A_{\widetilde f}) \to \pi_{i}(M - A_{f})$ and $q_{|*} \co \pi_1(A_{\widetilde f}) \to \pi_1(A_f)$ be the homomorphisms induced by the restrictions $q_{|} \co N - A_{\widetilde f} \to M - A_{f}$ and $q_| \co A_{\widetilde f} \to A_f$ (we consider the collection of these homomorphisms when $A_{\widetilde f}$ is not connected).

The following proposition is simple and the proof is left to the reader.

\begin{proposition}
	Suppose that $q(\overline N(A_{\widetilde f})) \subset \overline N(A_f)$.
	We have $q_{*}(\mu_1(\overline N(A_{\widetilde f}), A_{\widetilde f})) \subset \mu_1(\overline N(A_f), A_f)$, $\omega_{\widetilde f} = \omega_{f} \circ q_{|*}$, $\widehat\omega_{\widetilde f} = \widehat\omega_{f} \circ q_{|*}$, and $\omega^s_{\widetilde f} = \omega^s_{f} \circ q_{|*}$. Moreover, the singular bundle of $\widetilde f$ is the pullback of the singular bundle of $f$ by $q_{|A_{\widetilde f}}$, hence $\omega_{\widetilde f}^\Join = \omega_f^\Join \circ q_{|*}$.
\end{proposition}

\paragraph{The twisting operation}

Consider a Lefschetz fibration $f \co V \to M$ with exceptional fiber $F$. Let $\psi \in \Pi_1(F)$. We are going to construct a new Lefschetz fibration $f_{\psi} \co V_\psi \to M$. Consider an oriented 2-disk $D \subset M - A_f$, and take a tubular neighborhood $C \times B^{m-1}$ of $C = \partial D$, with $m = \dim M$. Fix the (isotopically unique) trivialization of $f$ over $C \times B^{m-1}$ that extends over $D$. This determines a fiberwise embedded copy of $C \times B^{m-1} \times F$ in $V$, that is the preimage of $C \times B^{m-1}$. Now, twist $f$ over $C$ by means of $\psi$. To do this, remove $\Int(C \times B^{m-1} \times F)$ from $V$, and glue it back differently by composing the original attaching diffeomorphism to the right with $\Psi \co C \times B^{m-1} \times F \to C \times B^{m-1} \times F$ which is defined by $\Psi(x, y, z) = (x, y, \psi_x(z))$, where (up to some identifications) $\psi \co C \to \diff(F)$, $\psi \co x \mapsto \psi_x$, satisfies $\psi_{x_0} = \id_F$ for some $x_0 \in C$. 

What we get is a new Lefschetz fibration $f_\psi \co V_\psi \to M$. We call $f_\psi$ the {\em twisting} of $f$ by $\psi$. The twisting operation has been considered by Moishezon in \cite{M77} as one of the main tools needed to classify positive genus-1 Lefschetz fibrations over the 2-sphere. See also, for example, \cite{KMMW05} for a similar usage in the context of achiral fibrations.

\begin{remark}
	By results of Moishezon \cite[Part II]{M77} the twisting of $f$ is equivalent to $f$ if $\omega_f$ is surjective.
\end{remark}

\begin{theorem}
	Suppose is given the datum $L = (M, A, F, \omega, \widehat\omega, \zeta)$, where $A \subset M$ is a co\-di\-men\-sion-2 submanifold, $F$ is a not exceptional connected surface, $\omega \co \mu_1(\overline N(A), A) \allowbreak \to \M(F)$ and $\widehat\omega \co \pi_1(M - A) \allowbreak \to \widehat\M(F)$ are two homomorphisms that fit in the commutative diagram \eqref{cd1/eqn}, and $\zeta$ is a bundle over $A$ with fiber $F/c$, $c \subset F$ a simple curve that depends on the component of $A$, such that $\zeta$ is compatible with $\omega$ and $\widehat\omega$ in the above sense. Then there exists a Lefschetz fibration $f_L \co V_L \to M_L$ with fiber $F$, uniquely determined by $L$ up to equivalence, such that $A_{f_L} = A$, $\omega_{f_L} = \omega$, $\widehat\omega_{f_L} = \widehat\omega$, and having singular bundle equivalent to $\zeta$. Moreover, for another datum $L' = 
	(M', A', F', \omega', \widehat\omega', \zeta')$ we have $f_L \cong f_{L'}$ if and only if  there are diffeomorphisms $\psi \co (M, A, *) \to (M', A', *')$ sending $\overline N(A)$ onto $\overline N(A')$, and $h \co F \to F'$ such that $(1)$ $\zeta \cong \zeta'$ by a bundle equivalence that covers $\psi_{|A} : A \to A'$, and $(2)$ $\omega' \circ \psi_{1*} = h_* \circ \omega$ and $\widehat\omega' \circ \psi_{2*} = h_* \circ \widehat\omega$, where the $\psi_{i*}$ are the isomorphisms induced, respectively, by $\psi_1 = \psi_| \co \overline N(A) - A \to \overline N(A') - A'$ and $\psi_2 = \psi_| \co M - A \to M' - A'$ on the fundamental group, and $h_*$ is the canonical isomorphism induced by $h$ between the relevant mapping class groups $h_* \co \M(F) \cong \M(F')$ or $h_* \co \widehat\M(F) \cong \widehat\M(F')$. If $F$ is exceptional, the same holds up to twistings.
\end{theorem}

Actually, this is a consequence of known general facts in fiber bundles theory, so we give only a sketch of the proof.

\begin{proof}[Proof (sketch)]
	By the classical theory of fiber bundles, $\widehat\omega$ determines uniquely an $F$-bundle over $M - \Int N(A)$, see for example \cite[Chapter 5]{FM12}. On the other hand, $\omega$ and $\zeta$ determine a Lefschetz fibration over $\overline N(A)$. Glue these fibrations by a suitable fibered diffeomorphism. This proves the existence. For the unicity, notice that any such fibered diffeomorphism extends to the interior of $N(A)$. Indeed, this is well-known in dimension two, and working on the tubular neighborhood $N(A)$ thought as a disk bundle $B^2 \hookrightarrow N(A) \to A$, one can adapt the two-dimensional case in a fiberwise fashion. Of course, the only ambiguity occurs when $F$ is exceptional, and this can be handled by a suitable twisting. 
\end{proof}

Note that the twisting action of $\Pi_1(F)$ is transitive on the set of possible structure monodromies for a fixed $(M, A, F, \omega, \widehat\omega)$. However, the structure monodromy cannot be used to resolve the ambiguity of the twisting action, as it can be easily seen by considering genus-1 Lefschetz fibrations over a closed surface.

\paragraph{Hurwitz systems and the monodromy sequence} By a Hurwitz system for a co\-di\-men\-sion-2 submanifold $A \subset M$ we mean a sequence $(\xi_1, \dots, \xi_n; \eta_1,\dots, \eta_k)$ where $\{\xi_1, \dots, \xi_n\}$ are meridians of $A$ that normally generate $\mu_1(\overline N(A), A)$, and $\{\eta_1, \dots, \eta_k\}$ are generators for $\pi_1(M - A)$ which are non-trivial in $\pi_1(M)$.

Once a Hurwitz system is fixed, we can represent the Lefschetz and the bundle monodromies of $f$ by a sequence of Dehn twists and mapping classes $(\delta_1, \dots, \delta_n; \allowbreak \gamma_1, \dots, \allowbreak \gamma_k)$, that we call the {\em monodromy sequence} of $f$. The elements of the monodromy sequence are given by $\delta_i = \omega_f(\xi_i) \in \M_{g,b}$, and $\gamma_i = \widehat\omega_f(\eta_i) \in \widehat\M_{g,b}$. 

Note that the Dehn twist $\delta_i$ is determined by a curve $c_i \subset F$, and by its sign. The curves $(c_1, \dots, c_n)$ are the vanishing cycles of $f$ with respect to the given Hurwitz system.

\section{The characterization theorems}

We denote by $\cal C_{g,b}$ the finite set of equivalence classes of homologically essential curves in $F = F_{g,b}$ up to orientation-preserving diffeomorphisms of $F$. Note that $\#\cal C_{g,b} = 1$ if $b \in \{0, 1\}$.

\begin{theorem}\label{universal/thm}
	A Lefschetz fibration $u \co U \to M$ with regular fiber $F$ is universal with respect to the class of Lefschetz fibrations over a surface and with fiber $F$, if the following three conditions hold:
	
	\begin{enumerate}
		\item
			$\widehat\omega_u$ is an isomorphism;
		
		\item
			$\omega_u$ and $\omega_u^s$ are surjective;
			
		\item
			any class of $\cal C_{g,b}$ can be represented by a vanishing cycle of $u$.
	\end{enumerate}
	
	On the other hand, as a partial converse, $u$ being universal implies $(2)$, $(3)$, and the surjectivity of $\widehat\omega_u$. 
	
	In particular, for $g \geq 2$ and $b \in \{0, 1\}$, $u$ is universal if $\widehat\omega_u$ is an isomorphism and $\omega_u$ is surjective.
	
	It follows that there exist universal Lefschetz fibrations for any fiber.
\end{theorem}

The surjectivity of $\omega^s_u$ means that any locally trivial $F_{g,b}$-bundle over $S^{2}$ is the pullback of $u$ by a map $S^{2} \to M - A_{u}$.

\begin{remark}
	If $u$ is universal we cannot conclude that $\widehat\omega_u$ is an isomorphism. The reason is that any Lefschetz fibration can be embedded in a larger Lefschetz fibration by preserving the universality. For example we can add a 1-handle $H^1$ to the base (if it has boundary), along with a fiberwise attachment of $H^1 \times F$ to the total manifold. So, we can add non-trivial elements to $\ker \widehat\omega_u$.
\end{remark}

This theorem generalizes the following proposition, which has been proved in \cite{Z11}.

\begin{proposition}\label{main-old/thm}
	A Lefschetz fibration $u \co U \to S$ over a surface with regular fiber $F$ is universal with respect to bounded base surfaces if and only if the following two conditions are satisfied:
	\begin{itemize}
		\item [$(1)$]
			$\omega_u$ and $\widehat\omega_u$ are surjective;
			
		\item [$(2)$]
			any class of $\cal C_{g,b}$ can be represented by a vanishing cycle of $u$.
	\end{itemize}
\end{proposition}

\begin{proof}[Proof of Theorem \ref{universal/thm}]
	Suppose that $u \co U \to M$ satisfies the conditions of the statement and let $f \co V \to S$ be a Lefschetz fibration with regular fiber $F_{g,b}$ over a surface $S$. 
	
	If $S$ is closed let $S' \subset S$ be the complement of a disk in $S$, with $A_f \subset \Int S'$, and let $f' \co V' = f^{-1}(S') \to S'$ be the restriction of $f$ over $S'$. Otherwise, if $S$ has boundary, put $S' = S$ and $f' = f$.
	
	We claim that there is a connected surface $G \subset M$ transverse to $A_u$, such that $G \cap \overline N(A_u)$ is connected, the meridians of $A_u$ that are contained in $G \cap \overline N(A_u)$ normally generate $\mu_1(\overline N(A_u), A_u)$, and $G$ is $\pi_1$-surjective in $M$. 
	
	We start the construction of $G$ by considering a 2-disk $G_0 \subset \overline N(A_u) - A_u$ centered at $*$. Then we attach 2-dimensional bands $G_1, \dots, G_n \subset \overline N(A_u)$, each one representing a meridian of $A_u$ so that they normally generate $\mu_1(\overline N(A_u), A_u)$. The band $G_i$ is attached to $G_0$ along an arc for each $i \geq 1$. Then we attach suitable 2-dimensional orientable 1-handles to $G_0$ (chosen to be disjoint from $A_u$) which realize a finite set of generators for $\pi_1(M)$. The resulting surface satisfies the conditions of the claim.
	
	Now consider the Lefschetz fibration $u' = u_| \co U' \to G$ which is the restriction of $u$ over $G$, with $U' = u^{-1}(G) \subset U$. It turns out that $u'$ satisfies the conditions (1) and (2) of Proposition~\ref{main-old/thm}, hence $u'$ is universal for Lefschetz fibrations over bounded surfaces. Then $f' \cong (q')^*(u') = (q')^*(u)$ for a $u$-regular map $q' \co S' \to G \subset M$.
	
	The loop $\beta = q'(\Bd S')$, homotoped to represent an element of $\pi_1(M - A_u)$, satisfies $\widehat\omega_u(\beta) = 1$. Therefore, $\beta$ is trivial in $\pi_1(M - A_u)$ because $\widehat\omega_u$ is an isomorphism. So, the map $q'$ extends to a $u$-regular map $q \co S \to M$ such that $q(S - S') \subset M - A_u$.
	
	Now, if $F$ is not exceptional, we immediately conclude that $f = q^*(u)$, proving that $u$ is universal.
	
	Otherwise, if $F$ is exceptional, $q^*(u)$ differs from $f$ by a twisting determined by an element $\psi \in \Pi_1(F)$. Since $\omega_u^s$ is surjective, there is a map $q'' \co S^2 \to M - A_u$ such that $\omega_u^s([q'']) = \psi$. 
	
	Up to a small homotopy relative to $q^{-1}(A_u)$, we can assume that there is a small disk $D \subset S$ such that $q_{|D} \co D \to M - A_u$ is an embedding. Similarly, we can assume that there is a small disk $D' \subset S^2$ such that $q''_{|D'}$ is an embedding.
	
	We can form the connected sum $q''' = q \,\#\, q'' \co S \,\#\, S^2 \cong S \to M$ by identifying $\partial D$ with $\partial D'$, and by connecting their images by a tube contained in $M - A_u$. It follows that $(q''')^{*}(u) \cong f$.
	
	\medskip
	
	\noindent {\em Proof of the (partial) converse.}
	Let $u \co U \to M$ be universal with fiber $F_{g,b}$. By letting $S$ to be a suitable surface with boundary, it can be easily constructed a Lefschetz fibration $f \co V \to S$ such that $\widehat\omega_f$ and $\omega_f$ are surjective, and such that any element of $\cal C_{g,b}$ can be represented by a vanishing cycle (meaning that there are sufficiently many critical points of $f$). Since $f = q^*(u)$ for some $u$-regualar map $q \co S \to M$, it immediately follows that $\widehat\omega_u$ and $\omega_u$ are surjective, and condition $(3)$ of the statement. 
	
	Regarding the surjectivity of $\omega_u^s$, this immediately follows by representing arbitrary $F_{g,b}$-bundles over $S^2$ by a pullback of $u$ (regarding a surface bundle as a Lefschetz fibration without critical points).
	
	Finally, the last sentence follows by the discussion in next section.
\end{proof}

In case of Lefschetz fibrations over 3-manifolds we have the following result.

\begin{theorem}\label{universal3/thm}
	Let $u \co U \to M$ be a Lefschetz fibration with fiber $F$ which satisfies the following conditions:
	\begin{enumerate}
		\item
			$\widehat\omega_u$ and $\omega_u^s$ are isomorphisms (so, $\pi_2(M - A_u) = 0$ for $F$ not exceptional);
		
		\item
			$\omega_u$ and $\omega_u^\Join$ are surjective;
			
		\item
			any class of $\cal C_{g,b}$ can be represented by a vanishing cycle of $u$;
		
		\item
			$A_u$ is connected.
	\end{enumerate}
	Then $u$ is universal for Lefschetz fibrations over 3-manifolds and with fiber $F$ .
\end{theorem}

\begin{proof}
	Let $f \co V \to Y$ be a Lefschetz fibration, with $Y$ a connected 3-manifold. We want to show that $f$ is a pulback of $u$. The critical image $L = A_f$ is a curve in $Y$, that is a disjoint union of circles and arcs.
	
	\begin{claim} 
		If $Y$ is closed there is a handle decomposition of the form $$Y = H^0 \cup n_1 H^1 \cup n_2 H^2 \cup H^3$$ such that: 
		\begin{enumerate}
			\item
				$H^0 \cap L$ is a possibly empty set of trivial arcs;
				
			\item 
				$H^1 \cap L$ is either empty or the core of $H^1$ for any 1-handle;
				
			\item
				$H^i \cap L = \emptyset$ for any higher index handle.
		\end{enumerate}
	\end{claim}
	
	\begin{proof}[Proof of the claim (sketch).] \renewcommand{\qedsymbol}{}
		Start from an arbitrary handle decomposition, with only one 0-handle and one 3-handle. Observe that, up to isotopy, we can assume that $L$ is disjoint from the 2- and the 3-handles, and that its intersection with any 1-handle is either empty or a number of parallel copies of its core. It is now straightforward to add new 1-handles and complementary 2-handles to normalize the intersections with the 1-handles. By adding canceling pairs of 1- and 2-handles again we can normalize also the intersection with the 0-handle, and this proves the claim.
	\end{proof}
	
	Now proceed with the proof of the theorem. First, by taking the double, we can assume that $Y$ is closed. Consider a handle decomposition of $Y$ as that of the claim.
	
	Over $H^0$, $f$ is a product $f_0 \times \id \co V_0 \times I \to B^2 \times I \cong H^0$, with $f_0 \co V_0 \to B^2$ a Lefschetz fibration. It follows that $f_0$ is a pullback of $u$, because $u$ is universal for Lefschetz fibrations over a surface by Theorem~\ref{universal/thm}. So, there is a $u$-regular map $q \co H^0 \to M$ such that $q^*(u) = f_{|H^0}$.
	
	Next, we extend this map $q$ handle by handle, and after each step we continue to denote by $q$ also the extension. If $H^1$ does not intersect $L$, the monodromy of a loop that passes through it geometrically once can be easily realized by a map to $M - A_u$ that extends $q$ because $\widehat\omega_u$ is surjective, and this map trivially extends over the 1-handle.
	
	If $H^1$ intersects $L$, we can find an arc in $A_u$ between the two endpoints $q(S^0 \times \{0\})$, where $H^1 = B^1 \times B^2 \supset S^0 \times \{0\}$. This arc can be suitably chosen to realize the singular monodromy of $f$ along the core of $H^1$, by using the fact that $A_u$ is connected and $\omega_u^\Join$ is surjective. This means that $q$ can be extended over the core of $H^1$, hence to $H^1$.
	
	Extending $q$ to the 2-handles is possible because $\widehat\omega_u$ is an isomorphism. If $F$ is exceptional, we might also need to modify the map $q$ on $H^2$ in order to adjust the twisting, by an argument similar to that in the proof of Theorem~\ref{universal/thm}.
	
	Finally, extending $q$ to the 3-handle $H^3$ is also possible because over the attaching sphere $\Sigma$ of $H^3$, $f$ is a trivial bundle. So, $[q_{|\Sigma}] \in \ker \omega_u^s = 0$, and this implies that $q_{|\Sigma} \co \Sigma \to M - A_u$ is homotopic to a constant in $M - A_u$. Therefore, $q$ can be extended over $H^3$.
	
	We get $q \co Y \to M$ which is $u$-regular, such that $f = q^*(u)$.
\end{proof}

\section{Construction of universal Lefschetz fibrations}\label{construct/sec}

Now we give explicit constructions of universal Lefschetz fibrations. First, we handle the case of Lefschetz fibrations over a surface, and for the sake of simplicity we assume $b \in \{0,1\}$, although a similar construction can be made in general. Thereafter, we extend this construction to dimension three.

\paragraph{Dimension 2}
Consider a finite presentation of $\M_{g,b} = \langle \delta_1, \dots, \delta_k \, |\, r_1, \dots, r_l \rangle$ with generators $\delta_1, \allowbreak \dots, \allowbreak \delta_k$ and relators $r_1,\dots, r_l$. We assume that each $\delta_i$ corresponds to a positive or negative Dehn twist about a non-separating curve in $F_{g,b}$. Note that in this case $\#\cal C_{g,b} = 1$.

If $b = 1$, a presentation of $\widehat\M_{g,1}$ can be obtained from that of $\M_{g,1}$ by adding as a further relator the Dehn twist $r_0$ about a boundary parallel curve, expressed in terms of the generators $\delta_i$, that is $\widehat\M_{g,1} = \langle \delta_1, \dots, \delta_k \, |\, r_0, r_1 \dots, r_l \rangle$, where $r_0$ should be substituted by a product of the form $r_0 = \delta_{i_1}^{\epsilon_1} \cdots \delta_{i_p}^{\epsilon_p}$, with $i_j \in \{1, \dots, k\}$ and $\epsilon_j \in \{-1, 1\}$. Otherwise, if $b = 0$, we have $\widehat\M_{g,0} = \M_{g,0}$.

Now, consider a Lefschetz fibration $v \co V \to B^2$ with regular fiber $F_{g,b}$ and $k$ critical values, having  $(\delta_1, \dots,\allowbreak \delta_k)$ as the monodromy sequence with respect to some Hurwitz system. By abusing notation, we denote by $(\delta_1, \dots, \delta_k)$ also the elements of the Hurwitz system. That is, we consider $\pi_1(B^2 - A_v) = \langle \delta_1, \dots, \delta_k \rangle$.

By Proposition~\ref{main-old/thm}, $v$ is universal for Lefschetz fibrations with regular fiber $F_{g,b}$ over bounded surfaces.

Put $v' = \id \times v \co B^2 \times V \to B^2 \times B^2 \cong B^4$. Clearly $v'$ is a Lefschetz fibration with regular fiber $F_{g,b}$, and it is universal with respect to bounded base surfaces. Moreover $A_{v'} = B^2 \times A_v$ is a set of mutually parallel trivial disks in $B^4$. 

Each relator $r_i$ is a word in the generators $\delta_i$, so it can be represented by an embedded loop $\lambda_i$ in $S^3 - \partial A_{v'}$. Moreover, up to homotopy, we can assume that the loops $\lambda_i$ are pairwise disjoint.

Note that $\omega_{v'}$ and $\widehat\omega_{v'}$ are surjective.
In order to kill the kernel of $\widehat\omega_{v'}$ we add a 2-handle $H^2_i$ to $B^4$ along $\lambda_i$ with an arbitrary framing (for example with framing 0), for all $i$. Let $M_2$ be the resulting 4-manifold.

Let $L_i = H_i^2 \cap S^3$ be the attaching region of $H_i^2$.
Now, attach the trivial bundle $H_i^2 \times F_{g,b}$ to $B^2 \times V$ by a fiberwise identification $L_i \times F_{g,b} \cong (v')^{-1}(L_i)$. This is possible because $\lambda_i$ has trivial bundle monodromy. Let $U_2$ be the resulting 6-manifold.

We get a new Lefschetz fibration $u_2 \co U_2 \to M_2$ defined by $v'$ in $B^2 \times V \subset U_2$, and by the projection onto the first factor in $H^2_i \times F_{g,b}$ for all $i$.

If $F_{g,b}$ is not exceptional, by Theorem~\ref{universal/thm} we immediately conclude that $u_2$ is universal.

If $F$ is exceptional, in our situation we have $F = T^2$, and so $\Pi_1(T^2) \cong \Z^2$ \cite{Gr73}. The two generators of this group correspond to two oriented torus bundles $q_1$ and $q_2$ over $S^2$. Making the fiber sum of $u_2$ with $\id_{B^2} \times q_1$ and $\id_{B^2} \times q_2$ produces a Lefschetz fibration, which we still denote by $u_2 \co U_2 \to M_2$, that satisfies all the conditions of Theorem~\ref{universal/thm}, hence a universal one.

\paragraph{Dimension 3} 

Start with the Lefschetz fibration $v'$ of the above construction. First, we make the critical image connected.
Since the Dehn twists $\delta_i$ are conjugate to each other, we can add a suitable oriented band between the $i$-th and the $(i+1)$-th disks of $A_{v'}$, so that the monodromy extends over this band. After adding these bands for all $i \leq k-1$, we get a Lefschetz fibration $v'' \co V'' \to B^4$. Note that $A_{v''}$ is a ribbon disk in $B^4$.

At this point, we want to make the singular monodromy surjective. To do this, we modify also the base manifold as follows. Consider a finite set of generators for the general mapping class group of $F_{g,b}/c \cong F_{g-1, b}$ with two marked points, where $c$ is a vanishing cycle of $v''$, and take two points $a_1, a_2 \in \partial A_{v''}$, chosen to be very close to each other.

Let $g \co Y \to B^2$ be a Lefschetz fibration with fiber $F$, having 0 as the only critical value, with monodromy given by that of a meridian of $\partial A_{v''}$ in $S^3$ that bounds a disk in $S^3$ with center at the point $a_1$.

Add a 1-handle $H^1 = B^1 \times B^1 \times B^2 \cong B^1 \times B^3$ to $B^4$, with attaching sphere $\{a_1, a_2\}$. Then $v''$ extends over $H^1$ by the product $\id_{B^1} \times \id_{B^1} \times g \co B^1 \times B^1 \times Y \to H^1$. Actually, we attach $H^1$ trivially around $a_1$ and by realizing one generator of the mapping class group of $F/c$ around $a_2$. This is straightforward by taking into account the local product structure of $v''$ near $a_1$ and $a_2$. Proceed in a similar way to realize any generator by attaching further 1-handles.

We end with a Lefschetz fibration $v'''$ such that $A_{v'''}$ is connected (and of genus 0), the singular monodromy is surjective, and the Lefschetz and bundle monodromies are surjective. So, after adding suitable 2-handles, we make the bundle monodromy an isomorphism. 

If $F = T^2$ we have to make $\omega_{v'''}^s$ surjective, and this can be done by fiber sum with two torus bundles over the sphere, multiplied by the identity map, in analogy with the construction in dimension two.

We obtain a Lefschetz fibration $u' \co U' \to M'$ over a 4-manifold $M'$, which is universal for Lefschetz fibrations over surfaces. 

As the last step, we have to kill the kernel of $\omega_{u'}^s$. To do this, simply take a finite set of generators for $\ker(\omega_{u'}^s)$ and let $\alpha \co S^2 \to M' - A_{u'}$ be such a generator.

Consider the product $u'' = u' \times \id_{B^2} \co U' \times B^2 \to M' \times B^2$. In the 5-manifold $M' \times S^1 \subset \partial(M' \times B^2)$ the map $\alpha$ can be perturbed to an embedding, so it can be represented by an embedded sphere $\Sigma \subset M' \times S^1 - A_{u''}$. Add a 3-handle $H^3$ along this sphere, and extend $u''$ over $H^3$ by a trivial $F$-bundle. This is possible because $\Sigma$ is in the kernel of $\omega_{u''}^s$. Continue in this way to kill all the generators of the kernel. We end with a Lefschetz fibration $u_3 \co U_3 \to M_3$ over a 6-manifold $M_3$ which is universal for 3-dimensional bases.

\section{Lefschetz cobordism}\label{lf-cob/sec}

For a Lefschetz fibration $f \co V \to M$ we denote by $-f \co (-V) \to (-M)$ the same Lefschetz fibration between the same manifolds with reversed orientation. Note that $f$ and $-f$ have the same oriented fiber.
Let $f_1 \co V_1 \to M_1$ and $f_2 \co V_2 \to M_2$ be Lefschetz fibrations with fiber $F_g = F_{g,0}$ such that $\dim M_1 = \dim M_2 = m$, and with $M_i$ and $V_i$ closed.

\begin{definition} We say that $f_1$ and $f_2$ are cobordant if there exists a Lefschetz fibration $f \co W \to Y$ with the same fiber $F_g$ such that $\partial W = V_1 \sqcup (- V_2)$, $\partial Y = M_1 \sqcup (-M_2)$, and $f_{|\partial W} = f_1 \sqcup (-f_2) \co V_1 \sqcup (-V_2) \to M_1 \sqcup (-M_2)$. In particular, if $f_2 = \emptyset$, we say that $f_1$ is cobordant to zero or that it bounds.
\end{definition}

The cobordism of Lefschetz fibrations is clearly an equivalence relation. We denote by $\Lambda(g, m)$ the set of equivalence classes. We remark that we are considering only oriented, compact, not necessarily connected Lefschetz fibrations.

There is a general theory of (co)bordism in several flavours. The book of Conner and Floyd \cite{CL64} is a good reference for general bordism theory. On the other hand, in \cite{An08} it is considered the cobordism of maps having only singularities of some prescribed class specified by an invariant open subset of the space of $k$-jets. However, Lefschetz fibrations do not seem to fit well in this general setting, because of the rigidity of Lefschetz fibrations between closed manifolds. In \cite{St66} both the source and the target are allowed to change up to cobordism. In these theories there is no control over the fiber. However, to the author's knowledge, there is no a similar theory specific to Lefschetz fibrations.

\begin{definition}
	The sum of two cobordism classes is defined by $[f_1] + [f_2] = [f_1 \sqcup f_2 \co V_1 \sqcup V_2 \to M_1 \sqcup M_2]$.
\end{definition}

It turns out that this operation is well-defined (does not depend on the representatives), and $\Lambda(g,m)$ with this operation is an abelian group which we call the {\em Lefschetz cobordism group of genus $g$ and dimension $m$}. The identity element is the empty fibration (or equivalently, the class of a Lefschetz fibration that bounds), and the inverse is given by $-[f] = [-f]$. Indeed, $f - f$ bounds $f \times \id_I \co V \times I \to M \times I$.

We define another operation on $\Lambda(g,m)$. Let $D_i \subset M_i - A_{f_i}$ be a small ball, for $i = 1, 2$. So, $f_i$ is a trivial bundle over $D_i$, that is $f_i^{-1}(D_i)$ can be identified with $D_i \times F_g$. 

Let $M_1 \# M_2 = (M_1 - \Int(D_1)) \cup_\partial (M_2 - \Int(D_2))$ be the result of the identification $D_1 \cong -D_2$ restricted to the boundary, that is the ordinary connected sum. Also let $V_1 \#_{F_g} V_2 = (V_1 - \Int f_1^{-1}(D_1)) \cup_\partial (V_2 - \Int f_2^{-1}(D_2))$ be the result of the identification $f_1^{-1}(D_1) \cong D_1 \times F_g \cong -D_2 \times F_g \cong -f_2^{-1}(D_2)$, again restricted to the boundary. 

\begin{definition}
	The fiber sum of $f_1$ and $f_2$ is the Lefschetz fibration $f_1 \# f_2 \co V_1 \#_{F_g} V_2 \to M_1 \# M_2$ defined by $f_i$ on $V_i - \Int(f_i^{-1}(D_i))$, $i = 1, 2$.
\end{definition}

Note that, in general, the fiber sum operation depends on the choice of a gluing diffeomorphism $h \in \M_g$ that occurs in the above identification between the preimages of the balls. Actually, there is not a canonical choice for $h$. However, the following holds.

\begin{proposition}
	We have $[f_1] + [f_2] = [f_1 \# f_2]$. Therefore, the fiber sum does not depend on the choice of the attaching diffeomorphism up to cobordism. It follows that any class in $\Lambda(g, m)$ has a connected representative. 
\end{proposition}

\begin{proof}
	Take the product $(M_1 \sqcup M_2) \times I$, and add an orientable 1-handle $H^1 = B^1 \times B^m$ to it, with attaching region $(D_1 \sqcup D_2) \times \{1\}$. Also glue $H^1 \times F$ to $(V_1 \sqcup V_2) \times I$ along $(f_1^{-1}(D_1) \sqcup f_2^{-1}(D_2)) \times \{1\} \cong (D_1 \sqcup D_2) \times F_g$ with a fibered attaching diffeomorphism such that the fiber is mapped onto itself by the identity on $D_1 \times F_g$ and by the attaching diffeomorphism occurring in the fiber sum on $D_2 \times F_g$. 
	We get a cobordism between $f_1 + f_2$ and $f_1 \# f_2$, and this concludes the proof.
\end{proof}

There is an obvious forgetful homomorphism $\Phi \co \Lambda(g, m) \to \Omega_{m + 2}^{\SO} \oplus \Omega_m^{\SO}$ defined by $\Phi([f \co V \to M]) = ([V], [M])$. This is surjective on the second component, since for any $M$ there is the trivial fibration $M \times F_g \to M$.

Now we consider the case $m = 2$.
For $f \co V \to S$ let $n_+(f)$ be the number of positive critical points of $f$, and let $n_-(f)$ be the number of negative critical points.
There are two canonical homomorphisms $\sigma, \eta \co \Lambda(g, 2) \to \Z$, defined by $\sigma([f]) = \Sign(V)$, and $\eta([f]) = n_+(f) - n_-(f)$. 

\begin{proposition}
	$\sigma$ and $\eta$ are well-defined homomorphisms.
\end{proposition}

\begin{proof}
	It is obvious that $\sigma$ is well-defined and a homomorphism. Let us prove the proposition for $\eta$.
	Let $h \co W \to Y$ be a cobordism between $f_1 \co V_1 \to S_1$ and $f_2 \co V_2 \to S_2$. The critical image of $h$ is a properly embedded compact curve in the 3-manifold $Y$. So, $A_h$ is a disjoint union of circles and arcs. Circles do not contribute to $\eta$. If an arc has both endpoints in $S_1$ (or in $S_2$), these are two opposite critical points of $f_1$ (or $f_2$), so they cancel. If there is one endpoint in $S_1$ and the other in $S_2$, these are critical points of, respectively, $f_1$ and $f_2$ of the same sign. Since any critical point of $f_i$ is the endpoint of an arc, we get $\eta(f_1) = \eta(f_2)$, and so $\eta$ is well-defined. That $\eta$ is a homomorphism is immediate.
\end{proof}

\begin{remark}
	By results in \cite{EN05}, $\sigma$ and $\eta$ are surjective for $g \geq 2$. In fact, it is proved that any lantern relation contributes $\pm 1$ to the signature and (obviously) to $\eta$, so by putting sufficiently many lantern relations or its inverses, in the monodromy sequence of a Lefschetz fibration over the sphere, we realize all signatures and get the surjectivity of $\eta$. These fibrations are achiral.
\end{remark}

\begin{remark}
The signature defines an isomorphism $\Omega_4^{\SO} \cong \Z$. So, $\sigma$ is equivalent to the forgetful homomorphism $\Phi$.
\end{remark}

We conclude by showing a remarkable relation with the singular bordism groups, implying that $\Lambda(g, 2)$ and $\Lambda(g, 3)$ are finitely generated.

Let $\Omega_n(X)$ denote the $n$-dimensional singular bordism group of $X$. Recall that the elements of $\Omega_n(X)$ are the bordism classes of oriented singular $n$-manifolds in $X$, that is pairs of the form $(N, q)$, with $N$ a closed oriented $n$-manifold, and $q \co N \to X$. These groups can be expressed in terms of singular homology with coefficients in the cobordism ring $\Omega_*^{\SO}$, modulus odd torsion \cite{CL64}.

\begin{proposition}\label{cobord/thm}
	Let $f \co V \to M$ be a Lefschetz fibration with fiber $F_g$. For any $n$ there is a canonical homomorphism $f_* \co \Omega_n(M) \to \Lambda(g, n)$, defined by $f_*([(N, q)]) = [q^*(f)]$.
\end{proposition}

\begin{proof}
	Let $(N, q)$ be a representative of a class of $\Omega_n(M)$, so $q \co N \to M$. Up to a small homotopy we can assume that $q$ is $f$-regular. Then we can take the pullback $q^*(f)$.
	
	If $(N', q')$ is bordant to $(N, q)$, where $q' \co N' \to M$ is $f$-regular, there is a bordism  $Q \co Y \to M$, where $Y$ is a cobordism between $N$ and $N'$, and $Q_{|\partial Y} = q \sqcup (-q')$. Up to homotopy relative to the boundary, we can assume that $Q$ is $f$-regular. Then, $Q^*(f)$ is a cobordism between $q^*(f)$ and $(q')^*(f)$.
	
	It follows that the map $f_* \co \Omega_n(M) \to \Lambda(g, n)$, $f_*([(N, q)]) = [q^*(f)]$, is a well-defined homomorphism.
\end{proof}

\begin{corollary}\label{lfcobord-bord/thm}
	If $u \co U \to M$ is universal with respect to Lefschetz fibrations over $n$-manifolds, then $u_* \co \Omega_n(M) \to \Lambda(g, n)$ is surjective. Therefore, there are epimorphisms $u_{n*} \co \Omega_n(M_n) \to \Lambda(g, n)$ for $n = 2, 3$, with $u_n \co U_n \to M_n$ the two universal Lefschetz fibrations constructed in Section~\ref{construct/sec}.
\end{corollary}

\begin{proof}
	Let $[f \co V \to N]$, $\dim N = n$, be an element of $\Lambda(g, n)$. Since $u$ is universal, there is $q \co N \to M$ such that $f = q^*(u)$, and so $[f] = u_*([(N, q)])$.
\end{proof}

Note that there is also the epimorphism $u_{3*} \co \Omega_2(M_3) \to \Lambda(g, 2)$.

\begin{corollary}
	There is an epimorphism $\lambda \co H_2(M_2) \to \Lambda(g, 2)$.
\end{corollary}

\begin{proof}
	The canonical homomorphism $\mu \co \Omega_2(M_2) \to H_2(M_2)$ defined by $\mu([(N, q)]) = q_*([N])$ is an isomorphism because $H_*(M_2)$ has no torsion and $H_i(M_2) = 0$ for $i \ne 0, 2$ (indeed, $M_2$ is $B^4$ union 2-handles) \cite[Chapter II]{CL64}. Therefore, $u_{2*} \circ \mu^{-1}$ is an epimorphism.
\end{proof}


\end{document}